\documentclass[12pt]{amsart}

\usepackage{amsfonts}
\usepackage{amssymb}
\usepackage{amsmath}

\begin{document}

\title {Note on the covering theorem for complex polynomials}
\author{Alexey Solyanik}
\date{24 October 2014 }

\email{transbunker@gmail.com}
\newtheorem{remark}{Remark}
\newtheorem{theorem}{Theorem}
\newtheorem{lemma}{Lemma}
\newtheorem{proposition}{Proposition}
\newtheorem{corollary}{Corollary}
\maketitle

In this note we will give the answer to the questions posed in \cite{dub} and \cite{dub2} concerning covering properties of complex polynomials.  These questions attract our attention when we study stability properties of dynamical systems stabilized by the feedback control. It was discovered recently \cite{dm1},  that these properties play the central role in the matter.

We start with the simple observation on the stable polynomials. From this observation we deduce the variant of Koebe One-Quarter Theorem for complex polynomials and from the last statement we obtain the sharp constant in the Theorem 1 from \cite{dub} (Theorem 3.8 in \cite{dub2} )

Let $\Delta=\lbrace z \colon \vert z \vert < 1\rbrace$ be the open unit disk and $\bar{\Delta}$ its closure. 

Let $\chi (z)= a_0 + a_1z + \dots a_n z ^n$ be a given polynomial of degree at most $n$. We would like to define it $n$-inverse by the formula
\begin{equation}
\label{chi}
\chi ^* (z)=a_n + a_{n-1}z + \dots +a_0z^n
\end{equation}
Since $\chi ^* (z) =z^n \chi  (1/z)$, polynomial $\chi ^* (z)$   "inverse"  the ranges $\chi(\Delta)$ and $\chi (\mathbb{C}\setminus \bar{\Delta})$.

We shall use this definition even in the case when $a_n=0$. For example the $4$-inverse of the polynomial $z$ is $z^3$ and vice versa.

So defined $n$-inversion is an \textit{involution } on the set of polynomials of degree at most $n$, i. e. $\chi^{**}(z)=\chi (z)$.

The problem of  \textit{stability} of polynomial $\chi(z)$, which means, that all roots of  $\chi(z)$ lie in $\Delta$, is equivalent to the problem of the description of the \textit{image }of $\bar{\Delta}$ by the map $\chi^*$ .

 Indeed,  if \textit{all} zeros of some polynomial $\chi(z)$ lie in the disc $\Delta$, then no one lies outside, which simply means, that $0 \notin \chi^*(\bar{\Delta})$.

To be more precise we write the last observation like a lemma.

\begin{lemma} Let $\chi (z)= a_0 + a_1z + \dots a_n z ^n$  be a polynomial with $a_n \neq 0$. Then
\begin{equation}
\label{g1}
\mbox{all zeros of} ~~~~~\chi (z)~~~~~~ \mbox{ lie in }~~~~ \Delta
\end{equation}
if and only if 
\begin{equation}
\label{g2}
0 \notin \chi^*(\bar{\Delta})
\end{equation}
\end{lemma}
\begin{proof}
Let $\chi(z)=a_nz^m(z-z_{m+1} )\dots (z-z_n)$, where $0 \leq m<n$ (the case $m=n$ is trivial). Then \begin{equation}
\label{g3}
\chi ^* (z) = z^n \chi (1/z)=a_n(1-z_{m+1}z)\dots (1-z_nz)
\end{equation}.

Thus, from (\ref{g1}) we conclude that all zeros of $\chi^* (z) $ lie outside of $\bar{\Delta}$, which implies (\ref{g2}).
On the other hand, if (\ref{g2}) holds, then (\ref{g3})    implies (\ref{g1}).
\end{proof}

The statement of  lemma 1 is also holds if we change $\Delta $ and  $\bar{\Delta} $ . We shall state here corresponding lemma in a little different (inversion) form and without proof, which is exactly the same.
\begin{lemma}
 Let $\chi (z)= a_0 + a_1z + \dots a_n z ^n$  be a polynomial with $a_0 \neq 0$. Then, 
\begin{equation}
\label{lyap2}
0 \notin \chi(\Delta)
\end{equation}
if and only if 
\begin{equation}
\label{lyap1}
\mbox{all zeros of} ~~~~~\chi^* (z)~~~~~~ \mbox{ lie in }~~~~ \bar{\Delta}
\end{equation}

\end{lemma}

It is a well known 
phenomenon  in Geometric Function Theory, 
that some  \textit{restrictions on the range $f(\Delta)$} implies \textit{estimates of Taylor coefficients of $f$}.  

For example, for
$$
f(z)=z+c_2z^2+ \dots +c_kz^k+ \dots 
$$ 
defined in the unit disk $\Delta$ :  if $f$ maps $\Delta$ \textit{into } the half-plane $\lbrace z: \Re z > -1/2 \rbrace$ then $|c_k|\leq 1$ (Caratheodori) and if the function is schlicht, i.e. maps $\Delta$ in one to one way to the simply connected domain,  then $|c_k| \leq k$ (de Brange).
 
The simplest proposition,  which reflect this phenomenon ( for complex polynomials) is the following

\begin{lemma}
\label{lemma_fc} Let $q(z)=\hat{q}(1)z +\hat{q}(2)z^2 +  \dots + \hat{q}(n) z^n$ and 
\begin{equation}
\label{l2}
w \notin q(\Delta).
\end{equation}
Then
\begin{equation}
\label{l3}
\vert \hat{q} (k) \vert \leq \binom {n} k \vert w \vert
\end{equation}

\end{lemma}
\begin{proof}
Let $\chi (z)=q(z)-w$ and $\chi^*(z)=\hat{q}(n) + \hat{q} (n-1)  z + \dots  + \hat{q} (1) z^{n-1} - w z ^n$  is it $n$-inverse. Then, $ 0 \notin \chi(\Delta)$ and, according to lemma 2, all roots of $\chi^*(z)=-w(z-\zeta_1) \dots (z-\zeta_n)$ lie in $\bar{\Delta}$. Hence, by  Vieta's formulas,

$$
\vert \hat{q}(k) \vert =\vert    \sum_{1 \leq i_1<i_2< \dots i_k \leq n } (-1)^{k+1} w  \zeta_{i_1 }\zeta_{i_2} \dots \zeta_{i_k}\vert \leq \vert w\vert \sum_{1 \leq i_1<i_2< \dots i_k \leq n } 1 = \binom {n} k \vert w \vert
$$

The estimates (\ref{l3}) are  the best possible, since for the polynomial $q(z)=w-w(1-z)^n$ point $w \notin q(\Delta)$, but $\vert \hat{q}(k) \vert
=\binom {n} k \vert w \vert$.
\end{proof}
\vskip .3cm
Define the \textit{norm} of $q(z)=\hat{q}(0) + \hat{q}(1)z + \dots \hat{q}(n) z^n$  by the formula
$$
n(q)= \max_{0 \leq k \leq n} \frac{ \vert \hat{q} (k)\vert }{\binom{n} k}
$$ 

Denote $\Delta_r (z_0)= \lbrace z \colon \vert  z - z_0 \vert < r \rbrace$ and claim, that, because  $q : \mathbb{C} \to \mathbb{C} $ is an open mapping, the range $q(\Delta_r (z_0))$ always contains some disk $\Delta_\epsilon  (q(z_0))$. 

It is interesting, that from lemma \ref{lemma_fc} we can  obtain the (sharp) estimate of  the radius of this disk and  immediately \textit{conclude} that $q(z)$ is an open mapping.
 
 It is clear, that we can suppose that $z_0=0$, $q(z_0)=0$ and $r=1$.

\begin{corollary}
\label{cor_radius_shell} For every polynomial $q(z)= \hat{q}(1) z + \hat{q}(2) z^2  +  \dots  + \hat{q}(n) z^n$ the range $q(\Delta)$ contains the disk of radius $n(q)$ with the center at the origin :
\begin{equation}
\label{radsh1}
\Delta_{n(q)} \subseteq q(\Delta).
\end{equation}
\end{corollary}
\begin{proof}
Lemma \ref{lemma_fc} implies $\vert w \vert \geq n(q)$ for $w \notin q(\Delta)$ and 
we have $\Delta_{n(q)} \subseteq q(\Delta)$.
\end{proof}

Since  $n(q) \geq 1/n$ for the polynomial $q$ of degree $n$, such that $q (0)=0$ and $q'(0)=1$ we have

\begin{corollary}
\label{cor_radius_shell_U} For every polynomial $q(z)=z + \dots q_n z^n$ the range $q(\Delta)$ contains the disk of radius $1/n$ with the center at the origin:
\begin{equation}
\label{radsh2}
\Delta_{1/n} \subseteq q(\Delta)
\end{equation}
\end{corollary}

If we apply Corollary \ref{cor_radius_shell_U} to the polynomial
 $$
 \tilde{q}(z)=\frac{1}{R}q(Rz)
 $$
 we can get a slight general version of this assertion which gives the answer to the question posed in \cite{dub} and \cite{dub2}. 
 %(p. 68 of Russian version).

\begin{corollary}
\label{cor_dub_koebe} For every polynomial $q(z)=z + \dots q_n z^n$ the range $q(\Delta_R)$ contains the disk of radius $\frac{R}{n}$ centered at the origin :
\begin{equation}
\label{dub1}
\Delta_{\frac{R}{n}} \subseteq q(\Delta_R)
\end{equation}
and
\begin{equation}
\label{dub2}
\bar{\Delta}_{\frac{R}{n}} \subseteq q(\bar{\Delta}_R)
\end{equation}
\end{corollary}

 The second statement (\ref{dub2}) 
% easily 
 follows from (\ref{dub1}) by the compactness arguments.
 Example 
 $$
 q(z) =\frac{R}{n}((1+\frac{z}{R})^n-1)
 $$ 
 shows that the size of the circle is the best possible.
 
Corollary 3  implies  the sharp constant in Theorem 1 from \cite{dub} .
  
 \begin{theorem} For every polynomial $p(z)=p_0 + p_1z + \dots + p_n z^n$ and any points $z_1$ and $z_2$ there exist point $\zeta$, such that $p(\zeta)=p(z_2)$ and 
\begin{equation}
\label{dub3}
\vert p(z_1) - p(z_2)  \vert \geq \frac{1}{n}\vert p'(z_1) \vert \vert z_1 -\zeta \vert
\end{equation}
\end{theorem}
\begin{proof}
Let
$$
q(z)=\frac{1}{p'(z_1)}(p(z_1) -p(z_1-z))
$$
Thus $q(0)=0$ and $q'(0)=1$. For any point $z$ there is a point $\eta$, such that $q(z)=q(\eta)$ and 
\begin{equation}
\label{dub3.1}
\vert q(z) \vert \geq \frac{1}{n}\vert \eta \vert 
\end{equation}
To prove this,  denote $w=q(z)$ and  define $R=n\vert w \vert$  for $w \neq 0$ ( if $w=0$ there is nothing to prove). We have $w \in \bar{\Delta}_{\frac{R}{n}}$ and Corollary \ref{cor_dub_koebe} implies that there is $\eta  \in  \bar{\Delta}_R$, such that $w=q(z)=q(\eta)$.

Now, for the given $z_2$ define $z=z_1-z_2$, apply (\ref{dub3.1}) and choose $\zeta = z_1-\eta$.

\end{proof}

\end{document}